\documentclass[runningheads,a4paper]{llncs}

\newcommand{\Gr}{Gr\"obner }

\usepackage{graphicx}
\graphicspath{{fig/}}

\usepackage{url}

\usepackage{amsmath}
\usepackage{amsfonts}
\usepackage{amssymb}
\usepackage{mathrsfs}
\usepackage{stmaryrd}

\newcommand{\cP}{\cal{P}}
\newcommand{\cJ}{\cal{J}}

\newcommand{\cM}{\cal{M}}
\newcommand{\cI}{\cal{I}}
\newcommand{\cR}{\cal{R}}
\newcommand{\cT}{\cal{T}}

\newcommand{\Q}{\mathbb{Q}}
\newcommand{\Z}{\mathbb{Z}}
\newcommand{\K}{\cal{K}}
\newcommand{\N}{\mathbb{N}}
\newcommand{\M}{\cal{M}}
\newcommand{\R}{\mathbb{R}}

\newcommand{\lm}{\mathop{\mathrm{lm}}\nolimits}

\newcommand{\lc}{\mathop{\mathrm{lc}}\nolimits}

\usepackage{algorithmic}

\newenvironment{algorithm}[1]{
  \begin{center}
    {\bf Algorithm: #1}\\*
    \begin{tabular}{|p{110mm}|} \hline
} {
 \\ \hline
 \end{tabular}
 \end{center}
}

\urldef{\mailsa}\path|gerdt@.jinr.ru|
\newcommand{\keywords}[1]{\par\addvspace\baselineskip
\noindent\keywordname\enspace\ignorespaces#1}

\begin{document}

\mainmatter              
\title{Consistency Analysis of Finite Difference Approximations to PDE Systems}
\titlerunning{Consistency Analysis of Finite Difference Approximations to PDE Systems}
\author{
Vladimir P. Gerdt}
\authorrunning{Vladimir P. Gerdt}
\institute{Laboratory of Information Technologies,\\ Joint Institute for Nuclear Research, 141980 Dubna, Russia\\ \mailsa
}

\maketitle              
\tocauthor{Vladimir P. Gerdt (JINR, Dubna)}

\begin{abstract}
In the given paper we consider finite difference approximations to systems of polynomially-nonlinear partial differential equations whose coefficients are rational functions over rationals in the independent variables.
The notion of strong consistency which we introduced earlier for linear systems is extended to nonlinear ones. For orthogonal and uniform grids we describe an algorithmic procedure for verification of strong consistency based on computation of difference standard bases. The concepts and algorithmic methods of the present paper are illustrated by two finite difference approximations to the two-dimensional Navier-Stokes equations. One of these  approximations is strongly consistent and another is not.
\end{abstract}

\section{Introduction}

Along with the methods of finite volumes and finite elements, the finite difference method~\cite{Samarskii'01} is widely used for  numerical solving of partial differential equations (PDE). This method is based upon the application of a local Taylor expansion to replace a differential equation by the difference one~\cite{Str'04,Th1'98} defined on the chosen computational grid. The last equation forms finite difference approximation (FDA) to the given PDE, and together with discrete approximation of initial or/and boundary condition constitutes a finite difference scheme (FDS).

In theory, the most essential feature required of discretization is convergence of a solution of FDS to a solution of PDE as the grid spacings go to zero. However, except a very limited class of problems, convergence cannot be directly analyzed. Instead, it has been universally adopted that convergence is provided if FDA is consistent and stable. This adoption is due to the brilliant Lax-Richtmyer equivalence theorem~\cite{Str'04,Th1'98} proved first for linear scalar PDE equations and then extended to some nonlinear scalar equations~\cite{Ros'82}. The theorem states that a consistent FDA to a PDE with the well-posed initial value (Cauchy) problem converges if and only if it is stable. Consistency implies reduction of FDA to the original PDE when the grid spacings go to zero. It is obvious that consistency is necessary for convergence. As to stability, it provides boundedness of the error in the solution under small perturbation in the numerical data.

Thus, the consistency check and verification of stability are principal steps in qualitative analysis of FDA to PDE. Modern computer algebra methods, algorithms and software may provide a powerful tool for generating FDA~\cite{GBM'06} and for performing its consistency and stability analysis. Some recent computer algebra application to study stability and to generate FDA to linear PDE systems with constant coefficients are discussed in~\cite{ML'11}. In papers~\cite{GB'09,GR'10} some computer algebra and algorithmic issues related to the consistency analysis were considered. In particular, for orthogonal and uniform solution grids the notion of s-consistency (strong-consistency) was introduced in~\cite{GR'10} for FDA to a linear PDE system that strengthens the conventional notion of consistency and admits algorithmic verification.  In doing so, an s-consistent discretization not only approximates the differential equations in a given linear system but also preserves at the discrete level algebraic properties of the system. It follows that if the system has local conservation laws in the form of algebraic consequences of its equations, then the s-consistent discrete system will also have such conservation laws (cf.~\cite{Dor'01,Th2'99}).

In this paper we generalize the concept of s-consistency to polynomially-nonlinear PDE systems and extend the algorithmic ideas of paper~\cite{GR'10} to check s-consistency for such systems on orthogonal and uniform solution grids. In the linear case algorithmic verification of s-consistency is based on completion of the initial differential system to involution and on construction of a \Gr basis for the linear difference ideal generated by FDA. It is important to emphasize that involutivity of the linear differential system under consideration not only makes possible an algorithmic verification of s-consistency but is also necessary (cf.~\cite{Seiler'10}) to well-posedness of Cauchy problem for the system what, if one believes in the extension of Lax-Richtmyer equivalence theorem to PDE systems, can provide convergence for s-consistent and stable FDA.

However, a differential system may not admit involutive form. Generally, one can decompose such a system into a finitely many involutive subsystems by applying the Thomas decomposition method~\cite{Thomas'37-62}. The  decomposition is done fully algorithmically~\cite{BGLHR'10} with the use of constructive ideas by Janet~\cite{Janet'29} further developed and generalized in~\cite{G'05,GB'98}. Another obstacle for nonlinear FDA is that the relevant nonlinear difference \Gr basis~\cite{Levin'08} may be infinite. Since it is commonly supposed that \Gr basis is a finite object, its infinite difference analogue is called standard basis as well  as in differential algebra (cf.~\cite{Ollivier'90,Zobnin'05}).

This paper is organized as follows. Section 2 contains a short description of differential and difference systems of equations which are studied in the paper. The properties of differential Thomas decomposition that are used for the s-consistency check are considered in Section 3. In Section 4 we define difference standard bases and present an algorithm for their construction. The definition of s-consistency of FDA for uniform and orthogonal grids, which is a generalization of that in~\cite{GR'10} to nonlinear differential systems, is given in Section 5. Here we also formulate and prove the main theorem on the algorithmic characterization of s-consistency and propose  an algorithmic procedure for its verification. The concepts and methods of the paper are illustrated in Section 6 by two FDA derived in~\cite{GB'09} for the two-dimensional Navier-Stokes equations. Some concluding remarks are given in Section 7.

\section{Preliminaries}

In the given paper we consider PDE systems of the form
\begin{equation}
f_1=\cdots=f_p=0,\quad F:=\{f_1,\ldots,f_p\}\subset {\cal{R}}\,.  \label{pde}
\end{equation}

Here $f_i$ $(i=1,\ldots, p)$ are elements in the {\em differential polynomial ring}  ${\cR}:={\K}[u^{1},\ldots,u^{m}]$, that is, polynomials in the dependent variables $\mathbf{u}:=\{u^{1},\ldots,u^{m}\}$ ({\em differential indeterminates}) and their partial derivatives
which are the operator power products of the derivation operators $\{\delta_1,\ldots,\delta_n\}$ $(\delta_j=\partial_{x_j})$. We shall assume that coefficients of the polynomials are rational functions in the independent variables $\mathbf{x}:=\{x_1,\ldots,x_n\}$ whose coefficients are rational numbers,
i.e. ${\K}:={\Q}(\mathbf{x})$.

To approximate the differential system (\ref{pde}) by a difference system we
shall use an orthogonal and uniform computational grid (mesh) as the set of
points $(k_1h_1,\ldots,k_nh_n)$
in $\R^n$. Here $\mathbf{h}:=(h_1,\ldots,h_n)$ $(h_i>0)$ is the set of
mesh steps (grid spacings) and the integer-valued
vector $(k_1,\ldots,k_n)\in \Z^n$ numerates the grid
points. If the actual solution to the problem (\ref{pde}) is the vector-function
$\mathbf{u}(\mathbf{x})$, then its approximation in the grid nodes will be
given by the grid (vector) function
$\mathbf{u}_{k_1,\cdots,k_n}=\mathbf{u}(k_1h_1,\ldots,k_nh_n)$.

We shall assume that coefficients of the differential polynomials in $F$ do not vanish in the grid points. The coefficients on the grid as rational functions in $\{k_1h_1,\ldots,k_nh_n\}$ are elements of the difference field~\cite{Levin'08} with mutually commuting {\em differences} $\{\sigma_1,\ldots,\sigma_n\}$ acting on a  function $\phi(\mathbf{x})$ as the right-shift operators
\begin{equation}
\sigma_i\circ
\phi(x_1,\ldots,x_n)=\phi(x_1,\ldots,x_i+h_i,\ldots,x_n)\quad (\,h_i>0\,).
\label{rs-operators}
\end{equation}

The {\em monoid} ({\em free commutative semigroup}) generated by $\mathbf{\sigma}$ will be denoted by $\Theta$, i.e.
\[
\Theta:=\{\,\sigma_1^{i_1}\circ\cdots\circ\sigma_n^{i_n}\mid i_1,\ldots,i_n\in \N_{\geq 0}\,\}\,,\qquad (\,\forall \theta\in \Theta\,)\ [\,\theta \circ 1=1\,]\,,
\]
the field of rational functions in $\{k_1h_1,\ldots,k_nh_n\}$ by $\tilde{\K}$ and the ring of {\em difference polynomials} over $\tilde{\K}$ by $\tilde{\cR}$. The elements in $\tilde{\cR}$ are polynomials in the dependent variables ({\em difference indeterminates}) $u^{\alpha}$ $(\alpha=1,\ldots,m)$ defined on the grid and in their shifted values $\sigma_1^{i_1}\circ\cdots\circ \sigma_n^{i_n}\circ u^\alpha$  $(i_j\in \N_{\geq 0})$. The coefficients of polynomials are taken from $\tilde{\K}$.

The standard technique to obtain FDA to (\ref{pde}) is replacement of the derivatives occurring in (\ref{pde}) by finite differences and application of appropriate power product of the right-shift operators (\ref{rs-operators}) to remove negative
shifts in indices which may come out of
expressions like
\[
\partial_j\circ u^{(i)}=\frac{u^{(i)}_{k_1,\ldots,k_j+1,\ldots,k_n}-u^{(i)}_{k_1,\ldots,k_j-1,
\ldots,k_n}}{2h_j}+O(h_j^2)\,.
\]
In \cite{GBM'06} we suggested another approach to generation of FDA which is
based on the finite volume method and on {\em difference elimination}. As it was shown for the classical Falkowich-Karman equation in gas dynamics, this method may derive a FDA which reveals better numerical behavior then those obtained by the standard technique.
In the sequel we shall consider FDA to the PDE system (\ref{pde}) as a finite set of difference polynomials
\begin{equation}
\tilde{f}_1=\cdots=\tilde{f}_q=0,\quad \tilde{F}:=\{\tilde{f}_1,\ldots,\tilde{f}_q\}\subset \tilde{\cR}\,, \label{fda}
\end{equation}
where $q$ need not be equal to $p$.

We shall say that a differential (resp. difference) polynomial $f\in {\cR}$ (resp. $\tilde{f}\in \tilde{\cR}$) is {\em differential-algebraic} (resp. {\em difference-algebraic}) {\em consequence} of (\ref{pde}) (resp. (\ref{fda})) if $f$ (resp. $\tilde{f}$) vanishes on the common solutions of (\ref{pde}) (resp. (\ref{fda})).

\section{Differential Thomas Decomposition}

\begin{definition} Let $S^{=}$ and $S^{\neq}$ be finite sets of differential polynomials
such that $S^{=}\neq \emptyset$ and contains equations $(\forall s\in S^{=})\ [s=0]$ whereas $S^{\neq}$ contains inequations $(\forall s\in S^{\neq})\ [s\neq 0]$.
Then the pair $\left(S^{=},S^{\neq}\right)$ of sets $S^{=}$ and $S^{\neq}$ is called {\em differential system}.
\end{definition}

Let $\mathfrak{Sol}\,(S^{=}/S^{\neq})$ denote the solution set of system $\left(S^{=},S^{\neq}\right)$, i.e. the set of common solutions of differential equations $\{\,s=0\mid s\in S^{=}\}$ that do not annihilate elements $s\in S^{\neq}$.

\begin{theorem}{\em\cite{Thomas'37-62}} Any differential system $\left(S^{=},S^{\neq}\right)$ is decomposable into a finite set of involutive subsystems $\left(S^{=}_i,S^{\neq}_i\right)$ with disjoint set of solutions
\begin{equation}
(S^{=}/S^{\neq})\Longrightarrow \underset{i} \bigcup \ (S^{=}_i/S^{\neq}_i)\,,\quad  \mathfrak{Sol}\,(S^{=}/S^{\neq})=\ \underset{i}\biguplus \ \mathfrak{Sol}\,(S^{=}_i/S^{\neq}_i)\,.
\label{decomposition}
\end{equation}
\label{DTD}
\end{theorem}

The structure of involutive subsystems in decomposition of a given system depends on the choice of {\em ranking} defined as follows. Consider the monoid of derivation operators $\Delta:=\{\,\delta_1^{i_1}\circ\cdots \circ\delta_n^{i_n}\mid i_1,\ldots,i_n\in \N_{\geq 0}\,\}$.

\begin{definition}
A total (linear) ordering $\succ$ on the set of partial derivatives $\{\delta\circ u^\alpha\mid \delta\in \Delta,\,\alpha=1,\ldots,\rho\}$ is {\em ranking} if for all $i,\alpha,\beta,\delta,\bar{\delta}$
\[
\delta_{i} \circ\delta u^\alpha \succ \delta \circ u^\alpha\,, \quad
\delta\circ u^\alpha \succ \bar{\delta}\circ u^\beta\quad
\Longleftrightarrow \quad \delta_i \circ \delta\circ  u^\alpha
   \succ \delta_i\circ \bar{\delta}\circ u^\beta\,.
\]
If\ \ $(\exists \gamma)\ [\delta\circ u^\gamma\succ \bar{\delta}\circ u^\gamma] \Longrightarrow (\,\forall\,\alpha,\beta\,)\ [\,\delta\circ u^{\alpha}\succ \bar{\delta}\circ u^{\beta}\,]$, then $\succ$ is {\em orderly}. If $u^\alpha\succ u^\beta\Longrightarrow (\,\forall\,\delta,\bar{\delta}\,)\ [\,\delta\circ u^{\alpha}\succ \bar{\delta}\circ u^{\beta}\,]$, then $\succ$ is {\em elimination}.
\label{dift-ranking}
\end{definition}

The Thomas decomposition into Janet involutive~\cite{Janet'29} subsystems is done fully algorithmically and have been implemented as a Maple package~\cite{BGLHR'10}. Given decomposition (\ref{decomposition}), one can algorithmically verify whether a differential equation $f=0$ $(f\in {\cal{R}})$ is a differential-algebraic consequence of the system $(S^{=},S^{\neq})$
\begin{equation}
(\, \forall a\in \mathfrak{Sol}\,(S^{=}/S^{\neq})\ \ [\,f(a)=0]\Longleftrightarrow (\,\forall\, i\,) \ [\,\text{dprem}_{\cJ}(f,S^{=}_i)=0\,]\,. \label{diff_conseq}
\end{equation}
Here $\text{dprem}_{\cJ}(f,P)$ denotes {\em differential Janet pseudo-reminder} of $f$ modulo $P$. The underlying Janet pseudo-division algorithm is described in~\cite{BGLHR'10} and implemented in the package.

\begin{remark}
For the case $S^{\neq}=\emptyset$ condition (\ref{diff_conseq}) verifies  $f\in \llbracket S^{=}\rrbracket\subset {\cal{R}}$, where $\llbracket F\rrbracket$ denotes the {\em radical} of differential ideal generated by the set $F$. Thereby, the Thomas decomposition of $(F,\emptyset)$ provides a {\em characteristic decomposition} of $\llbracket F\rrbracket$ (see~\cite{BGLHR'10,Hubert'01} for more details).
\label{rem:radical}
\end{remark}

\begin{example}
We illustrate the Thomas decomposition by the the example taken from \cite{G'08}. Consider differential system
\[
(\{(u_y+v)u_x+4v\,u_y-2v^2,\,(u_y+2v)u_x+5v\,u_y-2v^2\},\{\})
\]
with two quadratically-nonlinear first-order PDE with two dependent and two independent variables. Its Thomas decomposition for the ranking satisfying $u_x \succ u_y \succ v_x \succ v_y \succ u\succ v$ is given by
$$
  \left(
  \begin{array}{l}
  (u_y+v)u_x+4v\,u_y-2v^{\,2}\\
  u_y^{\,2}-3u_y+2v^{\,2}\\
  {v_x+v_y}
  \end{array},\,
   v \right) \ \bigcup \ \left(
  \begin{array}{l}
   u_x \\
   v
  \end{array},\, u_y \right) \ \bigcup \
  \left(
  \begin{array}{l}
   u_y\\
   v
   \end{array}\,, \emptyset \right)\,.
$$
\end{example}

For a differential system with linear PDEs and the empty set of inequations the decomposition algorithm performs  completion of the system to involution and returns the Janet basis form~\cite{Maple-Janet'03,G'99} of the input system.

\section{Difference Standard Bases}
For the shifted dependent variables {\em ranking} is defined in perfect analogy to Definition~\ref{dift-ranking} of ranking for partial derivatives.
\begin{definition}{\em\cite{Levin'08}}
A total ordering $\prec$ on $\{\,\theta\circ u^{\alpha}\mid \theta\in \Theta,\ 1\leq\alpha\leq m\,\}$ is {\em ranking} if for all  $\sigma_i,\theta,\theta_1,\theta_2,\alpha,\beta$
\[
(i)\,\sigma_i\,\circ\, \theta\,\circ u^\alpha \succ {\theta \circ u^\alpha}\,,\ \
(ii)\, {\theta_1 \circ u^\alpha} \succ {\theta_2 \circ u^\beta}\,
\Longleftrightarrow \, {
\theta}\circ {\theta_1\circ u^\alpha} \succ {\theta}\circ{\theta_2\circ u^\beta}\,.
\]
\label{def-diffc-ranking}
\end{definition}
\begin{definition}
A total ordering $\succ$ on the set ${\M}$ of {\em difference monomials}
\[
{\M}:=\{\,(\theta_1\circ u^{1})^{i_1}\cdots (\theta_m\circ u^{m})^{i_m}\mid \theta_j\in \Theta,\ i_j\in \N_{\geq 0},\ 1\leq j\leq m\,\}
\]
is {\em admissible} if it extends a ranking and satisfies
\[
(\forall\, t\in {\M}\setminus \{1\})\ [t\succ 1]\ \wedge\
(\,\forall\, \theta\in \Theta)\ (\,\forall\, t,v,w\in {\M}\,)\ [\,v\succ w\Longleftrightarrow t\cdot\theta\circ v\succ  t\cdot \theta\circ w\,].
\]
\label{mon-order}
\end{definition}

\begin{remark}
Similar to that in Definition \ref{dift-ranking} one can define {\em orderly} and {\em elimination} difference rankings. As an example of admissible monomial ordering we indicate the {\em lexicographical monomial ordering}  compatible with a ranking.
\label{lex-ord}
\end{remark}

Given an admissible ordering $\succ$, every difference polynomial $\tilde{f}$ has the {\em leading monomial} $\lm(\tilde{f})\in {\M}$ with the {\em leading coefficient} $\lc(\tilde{f})$. In what follows every difference monomial is to be {\em normalized (i.e. monic)} by division of the monomial by its leading coefficient. This  provides  $(\,\forall \tilde{f}\in \tilde{\cR}\,)\ [\,\lc(\tilde{f})=1\,]$.

Now we consider the notions of difference ideal~\cite{Levin'08} and its standard basis. The last notion is given here in analogy to that in differential algebra~\cite{Ollivier'90}.

\begin{definition}{\em\cite{Levin'08}} A set ${\cI}\subset {\tilde{\cR}}$ is {\em difference polynomial ideal}  or {\em $\sigma$-ideal} if
\[
(\,\forall\, a,b\in {\cI}\,)\ (\,\forall\, c\in \tilde{\cR}\,),\quad  (\,\forall\, \theta \in \Theta\,)\ [\,
a+b\in {\cI},\ a\cdot c\in {\cI},\ \theta\circ a\in {\cI}\,].
\]
If $\tilde{F}\subset \tilde{\cR}$, then the smallest $\sigma$-ideal containing $\tilde{F}$ is said to be generated by $\tilde{F}$ and denoted by $[\tilde{F}]$.
\label{diffc-ideal}
\end{definition}

If for $v,w\in {\M}$ the equality $w=t\cdot \theta\circ v$ holds with $\theta\in \Theta$ and $t\in {\M}$ we shall say that $v$ {\em divides} $w$ and write $v\mid w$. It is easy to see that this divisibility relation yields a {\em partial order}.

\begin{definition} Given a $\sigma$-ideal ${\cI}$ and an admissible monomial ordering $\succ$, a  subset $\tilde{G}\subset {\cI}$ is its {\em (difference) standard basis} if $[\tilde{G}]={\cI}$ and
\begin{equation}
(\,\forall\, \tilde{f}\in {\cI}\,) (\,\exists\, \tilde{g}\in \tilde{G}\,)\ \ [\,\lm(\tilde{g})\mid \lm(\tilde{f})\,]\,.
\label{SB}
\end{equation}
If the standard basis is finite it is called {\em \Gr basis}.
\label{def_SB}
\end{definition}

\begin{definition}
A polynomial $\tilde{p}\in \tilde{\cR}$ is said to be {\em head reducible modulo $\tilde{q}\in \tilde{\cR}$ to $\tilde{r}$} if $\tilde{r}=\tilde{p}-m\cdot\theta\circ \tilde{q}$ and $m\in {\cM}$, $\theta\in \Theta$ are such that $\lm(\tilde{p})=m\cdot\theta\circ \lm(\tilde{q})$. In this case transformation from
$\tilde{p}$ to $\tilde{r}$ is {\em elementary reduction} and denoted by ${\tilde{p}}\xrightarrow[\tilde{q}]{} \tilde{r}$. Given a set $\tilde{F}\subset \tilde{\cR}$, $\tilde{p}$ {\em is head reducible modulo $\tilde{F}$} $($denotation: ${\tilde{p}}\xrightarrow[\tilde{F}]{})$  if there is $\tilde{f}\in \tilde{F}$ such that $\tilde{p}$ is head reducible modulo $\tilde{f}$. A polynomial $\tilde{p}$ {\em is head reducible to $\tilde{r}$ modulo $\tilde{F}$} if there is a chain of elementary reductions
\begin{equation}
\tilde{p}\xrightarrow[\tilde{F}]{}\tilde{p}_1\xrightarrow[\tilde{F}]{} \tilde{p}_2\xrightarrow[\tilde{F}]{}\cdots \xrightarrow[\tilde{F}]{}\tilde{r}\,.
\label{red_chain}
\end{equation}
Similarly, one can define {\em tail reduction}. If $\tilde{r}$ in (\ref{red_chain}) and each of its monomials are not reducible modulo $\tilde{F}$, then we shall say that {\em $\tilde{r}$ is in the normal form modulo $\tilde{F}$} and write $\tilde{r}=\mathrm{NF}(\tilde{p},\tilde{F})$. A polynomial set $\tilde{F}$ with more then one element is {\em interreduced} if
\begin{equation}
(\,\forall \tilde{f}\in \tilde{F}\,)\ [\,\tilde{f}=\mathrm{NF}(\tilde{f},\tilde{F}\setminus \{\tilde{f}\})\,]\,. \label{interreduce}
\end{equation}
\label{reduction}
\end{definition}
Admissibility of $\succ$, as in commutative algebra,  provides termination of chain (\ref{red_chain}) for any $\tilde{p}$ and $\tilde{F}$. In doing so, $\mathrm{NF}(\tilde{p},\tilde{F})$ can be computed by the difference version of a multivariate polynomial division algorithm~\cite{BW'93,CLO'07}. If $\tilde{G}$ is a standard basis of $[\tilde{G}]$, then from Definitions \ref{def_SB} and \ref{reduction} it follows
\[
   \tilde{f}\in [\tilde{G}] \Longleftrightarrow \mathrm{NF}(\tilde{f},\tilde{G})=0\,.
\]
Thus, if an ideal has a finite standard (Gr\"{o}bner) basis, then its construction solves the ideal membership problem as well as in commutative~\cite{BW'93,CLO'07} and differential~\cite{Ollivier'90,Zobnin'05} algebra. The algorithmic characterization of standard bases, and their construction in difference polynomial rings is done in terms of difference $S$-polynomials.

\begin{definition}
Given an admissible ordering, and monic difference polynomials $\tilde{p}$ and $\tilde{q}$, the polynomial $S(\tilde{p},\tilde{q}):=m_1\cdot \theta_1\circ \tilde{p}-m_2\cdot \theta_2\circ \tilde{q}$
is called {\em $S$-polynomial} associated to $\tilde{p}$ and $\tilde{q}$ (for $\tilde{p}=\tilde{q}$ we shall say that {\em $S$-polynomial} is associated with $\tilde{p})$ if
$ m_1\cdot \theta_1\circ  \lm(\tilde{p})=m_2\cdot \theta_2\circ \lm(\tilde{q})$
with co-prime $m_1\cdot \theta_1$ and $m_2\cdot \theta_2$.
\label{S-polynomial}
\end{definition}

\begin{theorem}
Given an ideal ${\cI}\subset \tilde{\cR}$ and an admissible ordering $\succ$, a set of polynomials $\tilde{G}\subset {\cI}$ is a standard basis of ${\cI}$ if and only if $\mathrm{NF}(S(\tilde{p},\tilde{q}),\tilde{G})=0$ for
all $S$-polynomials, associated with polynomials in $\tilde{G}$.
\label{B-criterion}
\end{theorem}

\begin{proof}
It follows from Definitions \ref{def_SB}, \ref{reduction} and \ref{S-polynomial} in line with the standard proof of the analogous theorem for \Gr bases in commutative algebra~\cite{BW'93,CLO'07} and with the proof of similar theorem for standard bases in differential algebra~\cite{Ollivier'90}.  $\Box$
\end{proof}

Let ${\cI}=[\tilde{F}]$ be a $\sigma$-ideal generated by a finite set $\tilde{F}\subset \tilde{\cal{R}}$ of difference polynomials. Then for a fixed admissible monomial ordering the below algorithm \textsl{\bfseries{StandardBasis}}, if it terminates, returns a standard basis $\tilde{G}$ of ${\cI}$.  The subalgorithm \textsl{\bfseries{Interreduce}} invoked in line 11 performs mutual interreduction of the elements in $\tilde{H}$ and returns a set satisfying (\ref{interreduce}).

Algorithm \textsl{\bfseries{StandardBasis}} is a difference analogue of the simplest version of Buchberger's algorithm (cf.~\cite{BW'93,CLO'07,Ollivier'90}). Its correctness is provided by Theorem \ref{B-criterion}. The algorithm always terminates when the input polynomials are linear. If this is not the case, the algorithm may not terminate.  This means that the {\bf do while}-loop (lines 2--10) may be infinite as in the differential case~\cite{Ollivier'90,Zobnin'05}. One can improve the algorithm by taking into account Buchberger's criteria to avoid some useless zero reductions of line 5. The difference  criteria are similar to the differential ones~\cite{Ollivier'90}.

\begin{example}
Consider a simple example of the principal ideal generated by polynomial $\tilde{g}_1:=u(x)\cdot u(x+2)-x\cdot u(x+1)$ in the {\em ordinary} difference ring with the only shift operator (difference) $\sigma\circ u(x)=u(x+1)$, the independent variable (indeterminate) $u$ and the dependent variable $x$. Let us fix monomial ordering as the pure lexicographic one with $u(x)\prec u(x+1)\prec \cdots$. Obviously, it is admissible. Then a nontrivial (i.e. having nonzero normal form) $S$-polynomial $s_1$ associated with $\tilde{g}_1$ and its normal form $\tilde{g}_2$ modulo $\{\tilde{g}_1\}$ are given by
\begin{eqnarray*}
&& s_1:=u(x+4)\cdot \tilde{g}_1-u(x)\cdot \sigma^2 \circ \tilde{g}_1\,,\\
&& \tilde{g}_2:=\mathrm{NF}(s_1,\{\tilde{g}_1\})=u(x+1)\cdot u(x+4)-\frac{x+2}{x}\cdot u(x)\cdot u(x+3)\,.
\end{eqnarray*}
The second nontrivial $S$-polynomial $s_2$ associated with $\tilde{g}_1,\tilde{g}_2$ and its normal form $\tilde{g}_3$ modulo $\{\tilde{g}_1,\tilde{g}_2\}$ read
\begin{eqnarray*}
&& s_2:=u(x+4)\cdot \sigma \circ \tilde{g}_1 - u(x+3)\cdot \tilde{g}_2\,,\\
&& \tilde{g}_3:=\mathrm{NF}(s_2,\{\tilde{g}_1,\tilde{g}_2\})=u(x)\cdot u(x+3)^2-x\cdot (x+1)\cdot u(x+3) \,.
\end{eqnarray*}
One more nontrivial $S$-polynomial $s_3$ associated with $\tilde{g}_2,\tilde{g}_3$ and its normal
form $\tilde{g}_4$ modulo $\{\tilde{g}_1,\tilde{g}_2,\tilde{g}_3\}$ are
\begin{eqnarray*}
&& s_3:=\sigma \circ \cdot \tilde{g}_3 - u(x+4)\cdot \tilde{g}_2\,,\\
&& \tilde{g}_4:=\mathrm{NF}(s_3,\{\tilde{g}_1,\tilde{g}_2,\tilde{g}_3\})=u(x)\cdot u(x+3)\cdot u(x+4)-x\cdot (x+1)\cdot u(x+4)\,.
\end{eqnarray*}
The last nontrivial $S$-polynomial $s_4$ associated with $\tilde{g}_3,\tilde{g}_4$ and its normal
form $\tilde{g}_5$ modulo $\{\tilde{g}_1,\tilde{g}_2,\tilde{g}_3,\tilde{g}_4\}$ are
\begin{eqnarray*}
&& s_4:=u(x+5)\cdot \tilde{g}_3 - \sigma\circ \tilde{g}_4,,\\
&& \tilde{g}_5:=\mathrm{NF}(s_4,\{\tilde{g}_1,\tilde{g}_2,\tilde{g}_3,\tilde{g}_4\})=u(x+5)- \frac{x+3}{x\cdot(x+1)}u(x)\cdot u(x+4)\,.
\end{eqnarray*}
Now all $S$-polynomials associated with elements in $\tilde{G}:=\{\tilde{g}_1,\tilde{g}_2,\tilde{g}_3,\tilde{g}_4,\tilde{g}_5\}$ are reduced to zero modulo $\tilde{G}$, and $\tilde{G}$ is an interreduced standard basis of $[\tilde{g}_1]$.
\end{example}

\begin{algorithm}{\textsl{\bfseries{StandardBasis}}\,($\tilde{F},\succ$)\label{StandardBasis}}
\begin{algorithmic}[1]
\INPUT $\tilde{F}\in \tilde{\cal{R}}\setminus \{0\}$, a finite set of nonzero polynomials;\\ $\succ$, a monomial ordering \\
\OUTPUT $G$, an interreduced standard basis of $[F]$
\STATE $\tilde{G}:=\tilde{F}$
\DOWHILE
  \STATE  $\tilde{H}:=\tilde{G}$
  \FORALL{$S$-polynomials $\tilde{s}$ associated with elements in $\tilde{H}$}
    \STATE $\tilde{g}:=\mathrm{NF}(\tilde{s},\tilde{H})$
    \IF{$\tilde{g}\neq 0$}
      \STATE $\tilde{G}:=\tilde{G}\cup \{\tilde{g}\}$
    \ENDIF
  \ENDFOR
\ENDDO{$\tilde{G}\neq \tilde{H}$}
\STATE $\tilde{G}:=$\textsl{\bfseries{Interreduce}}\,($\tilde{G}$)
\RETURN $\tilde{G}$
\end{algorithmic}
\end{algorithm}

\section{Consistency of Finite Difference Approximations}

For simplicity, throughout this section we shall consider orthogonal and uniform grids with equisized mesh steps $h_1=\cdots=h_n=h$.

\begin{definition}{\em\cite{GR'10}} We shall say that a {\em difference equation $\tilde{f}(\mathbf{u})=0$ implies the differential equation $f(\mathbf{u})=0$} and write
$\tilde{f}\rhd f$ when the Taylor expansion about a grid point yields
\[
\tilde{f}(\mathbf{u})\xrightarrow[h\rightarrow 0]{} f(\mathbf{u})h^k + O(h^{k+1}),\ k\in \Z_{\geq 0}\,.
\]
\end{definition}

\begin{definition}{\em\cite{GR'10}}
Given a PDE system (\ref{pde}) and its difference approximation (\ref{fda}), we shall say that (\ref{fda}) is {\em weakly consistent} or {\em w-consistent} with (\ref{pde}) if
\[
   (\,\forall \tilde{f}\in \tilde{F}\,)\ (\,\exists f\in F\,)\ [\,\tilde{f}\rhd f\,]\,.
\]
\label{def-econs}
\end{definition}
In paper~\cite{GR'10} we showed that already for linear PDE systems such definition of consistency, which has been universally accepted in the literature, is not satisfactory in view of inheritance of properties of differential systems by their discretization. Instead, we introduced another concept of consistency for linear FDA which is extended to nonlinear systems of PDE as follows.

\begin{definition}{\em\cite{Levin'08}} A {\em perfect difference ideal} generated by set $\tilde{F}\in \tilde{\cal{R}}$ and denoted by $\llbracket \tilde{F}\rrbracket$ is the smallest difference ideal containing $\tilde{F}$ and such that for any $\tilde{f}\in {\cal{R}}$, $\theta_1,\ldots,\theta_r\in \Theta$ and $k_1,\ldots,k_r \in \N_{\geq 0}$
\[
(\theta_1\circ \tilde{f})^{k_1}\cdots (\theta_r\circ \tilde{f})^{k_r}\in \llbracket \tilde{F}\rrbracket \Longrightarrow \tilde{f}\in \llbracket \tilde{F}\rrbracket \,.
\]
\end{definition}

It is clear that $[\tilde{F}]\subseteq \llbracket \tilde{F}\rrbracket$. In difference algebra perfect ideals play the same role as radical ideals in commutative~\cite{CLO'07} and differential algebra~\cite{Hubert'01}, for example, in {\em Nullstellensatz}~\cite{Trushin'09}.  By this reason we shall consider the perfect ideal $\llbracket \tilde{F}\rrbracket$ generated by the difference polynomials
in FDA (\ref{fda}) as the set of its difference-algebraic consequences. Respectively, the set of differential-algebraic consequences of a PDE system is the radical differential ideal generated by the set $F$ in (\ref{pde}) (see  Remark~\ref{rem:radical}).

\begin{definition}
An FDA (\ref{fda}) to a PDE system (\ref{pde}) is {\em strongly consistent} or {\em s-consistent} if
\begin{equation}
(\,\forall \tilde{f}\in \llbracket \tilde{F} \rrbracket\,)\  (\,\exists
f\in \llbracket F \rrbracket \,)\ [\,\tilde{f}\rhd f\,]\,. \label{s-cond}
\end{equation}
\label{def-scon}
\end{definition}

The algorithm {\textsl{\bfseries{ConsistencyCheck}} presented below verifies s-consistency of  FDA to PDE systems. Its correction is provided by property (\ref{diff_conseq}) of the differential Thomas decomposition applied in lines 13--16 of the algorithm and by Theorem \ref{th:s-cons}. This theorem generalizes to nonlinear systems the theorem formulated and proved in~\cite{GR'10} for linear  systems.

\begin{theorem}
A difference approximation (\ref{fda}) to a differential system (\ref{pde}) is s-consistent if and only if a reduced standard basis $\tilde{G}\subset \tilde{\cR}$ of the difference
ideal $[\tilde{F}]$ satisfies
\begin{equation}
(\,\forall \tilde{g}\in \tilde{G}\,)\ (\,\exists g\in \llbracket F \rrbracket\,)\  [\,\tilde{g}\rhd g\,]\,. \label{cons-gb}
\end{equation}
\label{th:s-cons}
\end{theorem}

\begin{proof}
Let $\succ$ be an admissible monomial ordering and $\tilde{G}$ be the corresponding interreduced standard basis. To prove that (\ref{cons-gb}) implies (\ref{s-cond}) consider first a nonzero polynomial $\tilde{f}\in [F]$ and show that $\tilde{f}\rhd f\in \llbracket F\rrbracket$.  Polynomial $\tilde{f}$ as well as any $S$-polynomial associated with elements in $\tilde{G}$, because of the property (\ref{SB}) of $\tilde{G}$, admits representation with respect to $\tilde{G}$ and $\succ$ as a finite sum
\begin{equation}
\tilde{f}=\sum_{\tilde{g}\in \tilde{G}_1\subseteq \tilde{G}} \sum_{\mu} a_{\tilde{g},\mu}\cdot \sigma^\mu\circ {\tilde{g}}\,,\ \ a_{\tilde{g},\mu}\in \tilde{\cal{R}},\ \
\lm(a_{\tilde{g},\mu}\cdot \sigma^\mu\circ {\tilde{g}})\preceq \lm(\tilde{f})\,. \label{sb-ideal}
\end{equation}
Here we use the multiindex notation
\[
\mu:=(\mu_1,\ldots,\mu_n)\in \Z^n_{\geq 0},\ \sigma^\mu:=\sigma_1^{\mu_1}\circ\cdots\circ \sigma_n^{\mu_n}\,.
\]

Formula (\ref{sb-ideal}) is a difference analogue of the {\em standard representation} in commutative algebra~\cite{BW'93}. Consider the Taylor expansion (in grid spacing $h$) of the right-hand side of (\ref{sb-ideal})  about a grid point, nonsingular for the coefficients occurring in the sum. In doing so, the shift operators $\sigma_j$ $(j=1,\ldots,n)$ are expanded in the Taylor series
\begin{equation}
          \sigma_j=\sum_{k\geq 0}h^k\partial_j^k \label{sigma-expansion}
\end{equation}
along with the shifted coefficients as rational functions in the independent variables.

The representation (\ref{sb-ideal}) guarantees that in the leading order in $h$ the leading differential monomials~\cite{Ollivier'90}
which occur in the sum and come from
different elements of the \Gr basis cannot be cancelled out. Thereby, due to the
condition (\ref{cons-gb}), the Taylor expansion
of $\tilde{f}$ implies a finite sum of the form
\[
 f:=\sum_{g\in G_1} \sum_{\mu}b_{g,\nu}\cdot \partial^\nu\circ g,\quad b_{g,\nu}\in {\cR}\,,
\]
where $G_1:=\{g\in {\cR}\mid \exists \tilde{g}\in \tilde{G}_1\ \text{such that}\ \tilde{g} \rhd g\}$. Therefore, $\tilde{f} \rhd f\in [F]\subseteq \llbracket F \rrbracket$.

Let now $\tilde{p}\in \llbracket \tilde{F} \rrbracket \setminus [\tilde{F}]$ and $\theta_1,\ldots,\theta_r\in \Theta$ and $k_1,\ldots,k_r \in \N_{\geq 0}$ be such that
\begin{equation}
\tilde{q}:=(\theta_1\circ \tilde{p})^{k_1}\cdots (\theta_k\circ \tilde{p})^{k_r}\in [\tilde{F}]\,.
\label{shuffling}
\end{equation}
As we have shown, $\tilde{q}\rhd q\in [F]$, and it follows from (\ref{sigma-expansion}) that $q=p^{k_1+\cdots+ k_r}$ where $\tilde{p}\rhd p$. Hence, $p\in \llbracket F\rrbracket$. The perfect ideal $\llbracket \tilde{F} \rrbracket$ can be constructed \cite{Levin'08} from $[\tilde{F}]$ by the procedure called {\em shuffling} and based on enlargement of the generator set $\tilde{F}$ with all polynomials $\tilde{p}$ satisfying (\ref{shuffling}) and on repetition of such enlargement. It is clear that each such enlargement of the intermediate ideals yields in the continuous limit a subset of $\llbracket F \rrbracket$.

Conversely, conditions (\ref{cons-gb}) trivially follow from (\ref{s-cond}) and from $\tilde{G}\subset \llbracket \tilde{F} \rrbracket$. $\Box$
\end{proof}

\begin{algorithm}{\textsl{\bfseries{ConsistencyCheck}}\,($F,\tilde{F}$)\label{consistency}}
\begin{algorithmic}[1]
\INPUT $F\subset {\cal{R}}\setminus \{0\}$, $\tilde{F}\in \tilde{\cal{R}}\setminus \{0\}$, finite sets of nonzero polynomials  \\
\OUTPUT $\mathbf{true}$ if $\tilde{F}$ is s-consistent FDA to $F$, and $\mathbf{false}$    otherwise
\STATE {\bf choose} differential ranking $\succ_1$ and difference ordering $\succ_2$
\STATE ${\cT}:=$\textsl{\bfseries{DifferentialThomasDecomposition}}\,($F,\succ_1$)
\STATE ${\cP}_0:=\{\,P\mid \langle P,Q\rangle\in {\cT}\,\}$
\STATE $\tilde{G}:=$\textsl{\bfseries{StandardBasis}}\,($\tilde{F},\succ_2$)\qquad (* {\em may not terminate} *)
\STATE $C:=\mathbf{true}$
\WHILE{$\tilde{G}\neq \emptyset$\ and $C=\mathbf{true}$}
      \STATE {\bf choose} $\tilde{g}\in \tilde{G}$
      \STATE $\tilde{G}:=\tilde{G}\setminus \{\tilde{g}\}$;\ \ ${\cP}:={\cP}_0$
      \STATE {\bf compute} $g$ such that $\tilde{g}\rhd g$
      \WHILE{${\cP}\neq \emptyset$\ and $C=\mathbf{true}$}
          \STATE {\bf choose} $S\in {\cP}$
          \STATE ${\cP}:={\cP}\setminus \{S\}$
          \STATE $d:=$\textsl{\bfseries{dprem}}$_{\cJ}(g,S)$
          \IF{$d\neq 0$}
            \STATE $C:={\bf false}$
          \ENDIF
     \ENDWHILE
\ENDWHILE
\RETURN $C$
\end{algorithmic}
\end{algorithm}

It should be noted that condition (\ref{s-cond}) does not exploit the equality of cardinalities for sets of differential and difference equations as is assumed in Definition \ref{def-econs}. The equality of cardinalities is also not used in the proof of Theorem \ref{th:s-cons}. Therefore,
both Definition~\ref{def-scon} and Theorem \ref{th:s-cons} are relevant to the case when the FDA has the number of equations different from that in the PDE system.

In the nonlinear case when algorithm \textsl{\bfseries{StandardBasis}} may not terminate, it is useful to compute the continuous limit $\tilde{g}\rhd g$ for the difference polynomials $\tilde{g}$ obtained in line 5 of algorithm \textsl{\bfseries{StandardBasis}} and to verify the condition  \textsl{\bfseries{dprem}}$_{\cJ}(g,S)=0$ as it is done in line 14 of algorithm
\textsl{\bfseries{ConsistencyCheck}}. This way one can stop computation when   inconsistency of the intermedite data in algorithm \textsl{\bfseries{StandardBasis}} is detected. An example of such situation is considered in the next section.

\section{Example: Navier-Stokes Equations}

To illustrate the concept of s-consistency and the algorithmic procedure of its verification, we consider two FDA generated in~\cite{GB'09} for the two-dimensional Navier-Stokes equations by the method of paper~\cite{GBM'06}. These equations describe unsteady motion of incompressible viscous liquid of constant viscosity. The Janet involutive form of the Navier-Stokes equations for the orderly ranking compatible with
$\delta_x\succ \delta_y\succ \delta_t$ and  $u\succ v \succ p$ is given by (see~\cite{GB'09})
\begin{eqnarray}
F:=
\left\lbrace
\begin{array}{l}
f_1 := u_x+v_y=0\,,\\[0.1cm]
f_2 := u_t + u u_x + v u_y + p_x - \frac{1}{\mathrm{Re}} (u_{xx}+u_{yy})=0\,,\\[0.1cm]
f_3 := v_t + u v_x + v v_y + p_y - \frac{1}{\mathrm{Re}} (v_{xx}+v_{yy})=0\,,\\[0.1cm]
f_4 := u_x^2 + 2 v_x u_y + v_y^2 + p_{xx} + p_{yy} = 0 \,.
\end{array}
\right.\label{ns}
\end{eqnarray}
Here $f_1$ is the continuity equation, $f_2$ and $f_3$ are the proper
Navier-Stokes equations~\cite{Pozrikidis'01}, $f_4$ the pressure Poisson
equation~\cite{Gresho'87}, $(u,v)$ is the velocity field, and $p$ is the
pressure. The density is included in the Reynolds
number $\mathrm{Re}$.

The differential Thomas decomposition algorithm~\cite{BGLHR'10} for the input $f_1,f_2,f_3$ outputs system (\ref{ns}) in its Janet autoreduced form
\begin{eqnarray}
F_1:=
\left\lbrace
\begin{array}{l}
u_x+v_y=0\,,\\[0.1cm]
\frac{1}{\mathrm{Re}} (u_{yy}-v_{xy} - uv_y) - v u_y -u_t - p_x=0\,,\\[0.1cm]
\frac{1}{\mathrm{Re}} (v_{xx}+v_{yy}) - u v_x - v v_y -v_t - p_y=0\,,\\[0.1cm]
2 v_x u_y + p_{xx}+p_{yy} + 2 v_y^2 = 0 \,.
\end{array}
\right.\label{ins}
\end{eqnarray}

The following FDA to system (\ref{ns}) was obtained in~\cite{GB'09} for the orthogonal and uniform grid with the spatial spacing $h$ and temporal spacing $\tau$:
\begin{equation*}
\left\lbrace
\begin{array}{l}
\tilde{f}_1:=\frac{u^n_{j+1\, k} - u^n_{j-1\, k}}{2h} +
\frac{v^n_{j\, k+1} - v^n_{j\, k-1}}{2h} = 0
\,,\\[5pt]
\tilde{f}_2:=\frac{u^{n+1}_{j\, k} - u^{n}_{j\, k}}{\tau} +
\frac{{u^2}^n_{j+1\, k} - {u^2}^n_{j-1\, k}}{2h} +
\frac{{uv\,}^n_{j\, k+1} - {uv\,}^n_{j\, k-1}}{2h} +
{}\\[3pt] \qquad
+\frac{p^n_{j+1\, k} - p^n_{j-1\, k}}{2h} -
\frac{1}{\mathrm{Re}}\left(\frac{u^n_{j+2\, k} - 2u^n_{j\, k} + u^n_{j-2\, k}}{4h^2} +
\frac{u^n_{j\, k+2} - 2u^n_{j\, k} + u^n_{j\, k-2}}{4h^2}\right)=0\,,\\[5pt]
\tilde{f}_3:=\frac{v^{n+1}_{j\, k} - v^{n}_{j\, k}}{\tau} +
\frac{{uv\,}^n_{j+1\, k} - {uv\,}^n_{j-1\, k}}{2h} +
\frac{{v^2}^n_{j\, k+1} - {v^2}^n_{j\, k-1}}{2h} + {}\\[3pt] \qquad
+\frac{p^n_{j\, k+1} - p^n_{j\, k-1}}{2h} -
\frac{1}{\mathrm{Re}}\left(\frac{v^n_{j+2\, k} - 2v^n_{j\, k} + v^n_{j-2\, k}}{4h^2} +
\frac{v^n_{j\, k+2} - 2v^n_{j\, k} + v^n_{j\, k-2}}{4h^2}\right)=0\,,\\[5pt]
\tilde{f}_4:=\frac{{u^2}^n_{j+2\, k} - 2{u^2}^n_{j\, k} + {u^2}^n_{j-2\, k}}{4h^2} + 2\frac{{uv\,}^n_{j+1\, k+1} - {uv\,}^n_{j+1\, k-1} - {uv\,}^n_{j-1\, k+1} + {uv\,}^n_{j-1\, k-1}}{4h^2} +
{} \\[3pt] \qquad {} +
\frac{{v^2}^n_{j\, k+2} - 2{v^2}^n_{j\, k} + {v^2}^n_{j\, k-2}}{4h^2} +
\left(\frac{p^n_{j+2\, k} - 2p^n_{j\, k} + p^n_{j-2\, k}}{4h^2} +
\frac{p^n_{j\, k+2} - 2p^n_{j\, k} + p^n_{j\, k-2}}{4h^2}\right)=0\,.
\end{array}
\right.
\end{equation*}

This FDA is w-consistent what can be easily verified by the Taylor expansion of the difference polynomials in  $\tilde{F}:=\{\tilde{f}_1,\tilde{f}_2,\tilde{f}_3,\tilde{f}_4\}$ in the powers of $h,\tau$ about a grid point. In doing so, in the continuous limit $(\tau \rightarrow 0,\ h\rightarrow 0)$ the difference equations imply the involutive differential Navier-Stokes system (\ref{ns}). Moreover, the algorithm \textsl{\bfseries{StandardBasis}} applied to the set $\tilde{F}_1:=\{\sigma_y\circ\tilde{f}_1,\sigma_y\circ\tilde{f}_2,\tilde{f}_3,\tilde{f}_4\}$ yields that $\tilde{F}_1$ is a difference \Gr basis of ideal $[\tilde{F}_1]$ for the lexicographic ordering compatible with the orderly ranking such that $\sigma_t\succ\sigma_x\succ\sigma_y$ and $p\succ u\succ v$ (see Remark \ref{lex-ord}). Thus, $\tilde{F}$ is the s-consistent FDA to (\ref{ns}).

The above given FDA has a $5\times 5$ stencil owing to the approximation of the second-order partial derivatives used
for equations $f_2,f_3$ and $f_4$. From the numerical standpoint a $3\times 3$ stencil looks like  more attractive. By this reason let us follow~\cite{GB'09} and consider another FDA to (\ref{ns}) with a $3\times 3$ stencil:
\begin{equation*}
\left\lbrace
\begin{array}{l}
\tilde{e}_1:=\frac{u^n_{j+1\, k} - u^n_{j-1\, k}}{2h} +
\frac{v^n_{j\, k+1} - v^n_{j\, k-1}}{2h} = 0
\,,\\[5pt]
\tilde{e}_2:=\frac{u^{n+1}_{j\, k} - u^{n}_{j\, k}}{\tau} +
\frac{{u^2}^n_{j+1\, k} - {u^2}^n_{j-1\, k}}{2h} +
\frac{{uv\,}^n_{j\, k+1} - {uv\,}^n_{j\, k-1}}{2h} +
{}\\[3pt] \qquad
+\frac{p^n_{j+1\, k} - p^n_{j-1\, k}}{2h} -
\frac{1}{\mathrm{Re}}\left(\frac{u^n_{j+1\, k} - 2u^n_{j\, k} + u^n_{j-1\, k}}{h^2} +
\frac{u^n_{j\, k+1} - 2u^n_{j\, k} + u^n_{j\, k-1}}{h^2}\right)=0\,,\\[5pt]
\tilde{e}_3:=\frac{v^{n+1}_{j\, k} - v^{n}_{j\, k}}{\tau} +
\frac{{uv\,}^n_{j+1\, k} - {uv\,}^n_{j-1\, k}}{2h} +
\frac{{v^2}^n_{j\, k+1} - {v^2}^n_{j\, k-1}}{2h} + {}\\[3pt] \qquad
+\frac{p^n_{j\, k+1} - p^n_{j\, k-1}}{2h} -
\frac{1}{\mathrm{Re}}\left(\frac{v^n_{j+1\, k} - 2v^n_{j\, k} + v^n_{j-1\, k}}{h^2} +
\frac{v^n_{j\, k+1} - 2v^n_{j\, k} + v^n_{j\, k-1}}{h^2}\right)=0\,,\\[5pt]
\tilde{e}_4:= \frac{{u^2}^n_{j+1\, k} - 2{u^2}^n_{j\, k} + {u^2}^n_{j-1\, k}}{h^2} +
2\frac{{uv\,}^n_{j+1\, k+1} - {uv\,}^n_{j+1\, k-1} - {uv\,}^n_{j-1\, k+1} + {uv\,}^n_{j-1\, k-1}}{4h^2} +
{} \\[3pt] \qquad {} +
\frac{{v^2}^n_{j\, k+1} - 2{v^2}^n_{j\, k} + {v^2}^n_{j\, k-1}}{h^2} +
\left(\frac{p^n_{j+1\, k} - 2p^n_{j\, k} + p^n_{j-1\, k}}{h^2} +
\frac{p^n_{j\, k+1} - 2p^n_{j\, k} + p^n_{j\, k-1}}{h^2}\right)=0
\end{array}
\right.
\end{equation*}

$\tilde{F}':=\{\tilde{e}_1,\tilde{e}_2,\tilde{e}_3,\tilde{e}_4\}$ is w-consistent with (\ref{ns}). However, application of algorithm \textsl{\bfseries{StandardBasis}} shows that, as opposed to $\tilde{F}_1$, $\tilde{F}_1':=\{\sigma_y\circ \tilde{e}_1,\sigma_y\circ \tilde{e}_2,\tilde{e}_3,\tilde{e}_4\}$
it not a \Gr basis. For the $S$-polynomial $s_{1,2}$ associated with $\sigma_y\circ \tilde{e}_1$ and $\sigma_y\circ \tilde{e}_2$ we have $\tilde{q}:=\mathrm{NF}(s_{1,2},\tilde{F}_1')\neq 0$. Furthermore,
$
\tilde{q} \rhd q:=u^2_{xx}+v^2_{yy}+p_{xx}+p_{yy}\,.
$
The equation $q=0$ is not a consequence of the Navier-Stokes system.

One way to check it is to compute $d:=$\textsl{\bfseries{dprem}}$_{\cJ}(q,F_1)$ with $F_1$ given by (\ref{ins}). Just this computation is done in line 13 of algorithm {\textsl{\bfseries{ConsistencyCheck}}:
\[
\begin{array}{l}
d= \frac{1}{\mathrm{Re}^2}\left(u_{yy}^2 + v_{yy}^2-2 u_y v_x -2 v_y^2\right) + \frac{2}{\mathrm{Re}}\left(uv_yu_{yy} - v u_y u_{yy} - u_t u_{yy}- p_x u_{yy}\right) +  \\[0.1cm]
2\left(vu_t u_y -u u_t v_y + v u_y p_x - u v_y p_x - u v v_y u_y + u_t p_x\right) + u_t^2 + p_x^2 + v^2 u_y^2 + u^2 v_y^2\,.
\end{array}
\]

Another way is to substitute into $q$ the exact solution~\cite{KM'85} to (\ref{ns})
\[u=-e^{-2t}\cos(x)\sin(y), \
 v=e^{-2t}\sin(x)\cos(y),\
 p=-e^{-4t}(\cos(2x)+\cos(2y))/4\,.
\]
and to see that it does not satisfy $q=0$. Therefore, $\tilde{F}'$ is s-inconsistent.

\section{Conclusion}
Our computer experiments~\cite{GR'10} with linear systems based on the implementation~\cite{LDA'06} of Janet completion algorithm for the $\sigma$-ideals generated by linear difference polynomials shown that unlike w-consistency it is fairly difficult to satisfy s-consistency by discretizing overdetermined PDE systems. This is hardly surprising since an s-consistent FDA preserves at the discrete level all consequences of the differential system. As we demonstrate in Section 6 of the present paper, completion of the Navier-Stokes equations to involution by adding the Poisson pressure equation, which has to be explicitly taken into account in the numerical  solving~\cite{Gresho'87}, makes the s-consistency of their FDA sensitive to discretization.

To guarantee termination of the algorithmic versification of s-consistency one might use the fact that the difference polynomial ring we deal with in this paper is a Ritt ring and each its perfect ideal has a finite basis~\cite{Levin'08}. However, unlike the differential Ritt rings~\cite{Hubert'01}, there are no algorithms known to compute such basis and, hence, a \Gr basis for $\llbracket \tilde{F} \rrbracket$. Another obstacle in computer application to the consistency analysis of FDA is the lack of software for construction of nonlinear standard bases. Only very recently a start has been made with a new algorithmic insight inspired by the ideas of  paper~\cite{SL'11} with intention to create such software packages written in {\tt Maple} and {\tt Singular}\footnote{R. La Scala. Private communication.}.

\section{Acknowledgements}
The research presented in this paper was partially supported by grant 01-01-00200 from the Russian Foundation for Basic Research and by grant 3810.2010.2 from the Ministry of Education and Science of the Russian Federation. The author expresses his thanks to Yuri Blinkov, Viktor Levandovskyy, Alexander Levin and Roberto La Scala for helpful comments and remarks. 


\begin{thebibliography}{99}

\bibitem{BGLHR'10} B\"{a}chler, T., Gerdt, V.P., Lange-Hegermann, M., Robertz, D.: Thomas Decomposition of Algebraic and Differential Systems. In: Gerdt, V.P., Koepf, W., Mayr, E.W., Vorozhtsov, E.V. (eds.) CASC 2010. LNCS, vol. 6244, pp.\ 31--54. Springer, Berlin (2010) arXiv:math.AP/1008.3767

\bibitem{BW'93} Becker, T., Weispfenning, V.: \Gr Bases: A Computational Approach to
 Commutative Algebra. Graduate Texts in Mathematics, vol. 141. Springer, New York (1993)

\bibitem{Maple-Janet'03} Blinkov, Yu.A., Cid, C.F., Gerdt, V.P., Plesken, W., Robertz., D.: The MAPLE Package Janet: II. Linear Partial Differential Equations. Ganzha, V.G., Mayr, E.W., Vorozhtsov, E.V. (eds.)  Proceedings of the 6th International Workshop on Computer Algebra in Scientific Computing, pp.\ 41--54.  Technische Universit\"{a}t M\"{u}nchen (2003)\ \  Cf.\ also {\tt http://wwwb.math.rwth-aachen.de/Janet}

\bibitem{CLO'07} Cox, D., Little, J., O'Shie, D.: Ideals, Varieties and Algorithms. An
 Introduction to Computational Algebraic Geometry and Commutative Algebra. 3nd Edition.
 Springer, New York (2007)

\bibitem{Dor'01} Dorodnitsyn, V.: The Group Properties of Difference Equations.  Moscow, Fizmatlit (2001) (in Russian)

\bibitem{G'99} Gerdt, V.P.: Completion of Linear Differential Systems to Involution. Ganzha, V.G., Mayr, E.W., Vorozhtsov, E.V. (eds.) CASC'99. Computer Algebra in Scientific Computing /\ CASC'99, pp.\ 115--137.
   Springer, Berlin (1999) arXiv:math.AP/9909114

\bibitem{G'05} Gerdt, V.P.: Involutive Algorithms for Computing \Gr Bases. In: Cojocaru, S., Pfister, G., Ufnarovsky, V. (eds.) Computational Commutative and Non-Commutative Algebraic Geometry, pp.\ 199--225. IOS Press, Amsterdam (2005) arXiv:math.AC/0501111

\bibitem{G'08} Gerdt, V.P.: On Decomposition of Algebraic PDE Systems into Simple Subsystems. Acta Appl. Math. 101, 39--51 (2008)

\bibitem{GB'98} Gerdt, V.P., Blinkov, Yu.A.: Involutive Bases of Polynomial
 Ideals. Math. Comput. Simulat. 45, 519--542 (1998)  arXiv:math.AC/9912027

\bibitem{GB'09} Gerdt, V.P., Blinkov, Yu.A.: Involution and Difference Schemes for the Navier-Stokes Equations. Gerdt, V.P., Mayr, E.W., Vorozhtsov, E.V. (eds.) CASC 2009, LNCS, vol. 5743, pp.\ 94--105. Springer, Berlin (2009)

\bibitem{GBM'06} Gerdt, V.P., Blinkov, Yu.A., Mozzhilkin, V.V.:  \Gr Bases and Generation of
Difference Schemes for Partial Differential Equations. SIGMA 2, 051 (2006) arXiv:math.RA/0605334

\bibitem{LDA'06} Gerdt, V.P., Robertz, D.: A Maple Package for Computing \Gr Bases for Linear
Recurrence Relations. Nucl. Instrum. Methods 559(1), 215--219 (2006)   arXiv:cs.SC/0509070 \ \
Cf.\ also {\tt http://wwwb.math.rwth-aachen.de/Janet}


\bibitem{GR'10} Gerdt, V.P., Robertz, D.: Consistency of Finite Difference Approximations for Linear PDE Systems and its Algorithmic Verification. Watt, S.M. (ed.) Proceedings of ISSAC 2010, pp.\ 53--59. Association for Computing Machinery (2010)

 \bibitem{Hubert'01} Hubert, E.: Notes on Triangular Sets and Triangulation-Decomposition Algorithms. II: Differential Systems. Winkler, F., Langer, U. (eds.) SNSC 2001, LNCS, vol. 2630, pp.\ 40--87. Springer, Berlin (2001)

 \bibitem{Gresho'87} Gresho, P.M., Sani, R.L.: On Pressure Boundary Conditions for the
Incompressible Navier-Stokes Equations. Int. J. Numer. Meth. Fl. 7, 1111--1145 (1987)

\bibitem{Janet'29} Janet, M.: Le\c cons sur les Syst\`emes
 d'Equations aux D\'eriv\'ees Partielles. Cahiers Scientifiques, IV.
 Gauthier-Villars, Paris (1929)

\bibitem{KM'85} Kim, J., Moin, P.: Application of a Fractional-Step Method To Imcompressible Navier-Stokes Equations. J. Comput. Phys. 59, 308--323 (1985)

\bibitem{SL'11} La Scala, R., Levandovskyy, V.: Skew Polynomila Rings, \Gr Bases and The Letterplace Embedding of the Free Associative Algebra. arXiv:math.RA/0230289

\bibitem{Levin'08} Levin, A.: Difference Algebra. Algebra and Applications, vol. 8. Springer (2008)

\bibitem{ML'11} Martin, B., Levandovskyy, V.: Symbolic Approach to Generation and Analysis of Finite Difference
    Schemes of Partial Differential Equations. In: Langer, U., Paule, P. (eds.) Numerical and Symbolic Scientific Computing: Progress and Prospects, pp.123--156. Springer, Wien (2012)

\bibitem{Ollivier'90} Ollivier, F.: Standard Bases of Differential Ideals. Sakata, S. (ed.) AAECC-8. LNCS, vol. 508, pp.\ 304--321.  Springer, London (1990)

\bibitem{Pozrikidis'01} Pozrikidis, C.: Fluid Dynamics: Theory, Computation and Numerical Simulation. Kluwer, Amsterdam (2001)

\bibitem{Ros'82} Rosinger, E.E.: Nonlinear Equivalence, Reduction of PDEs to ODEs and Fast Convergent Numerical
Methods. Pitman, London (1983)

\bibitem{Samarskii'01} Samarskii, A.A.: Theory of Difference Schemes. Marcel Dekker, New York (2001)

\bibitem{Seiler'10} Seiler, W.M.: Involution: The Formal Theory of Differential Equations and its Applications in Computer Algebra. Algorithms and Computation in Mathematics 24. Springer, Heidelberg, (2010)

\bibitem{Str'04} Strikwerda, J.C.: Finite Difference Schemes and Partial Differential Equations, 2nd
Edition. SIAM, Philadelphia (2004)

\bibitem{Thomas'37-62} Thomas, J.M.: Differential Systems. AMS Colloquium Publications XX1 (1937); Systems and Roots. The Wylliam Byrd Press, Rychmond, Virginia (1962)

\bibitem{Th1'98} Thomas, J.W.:
Numerical Partial Differential Equations: Finite Difference Methods, 2nd Edition. Springer,
New York (1998)

\bibitem{Th2'99} Thomas, J.W.: Numerical Partial Differential Equations: Conservation Laws and Elliptic Equations. Springer, New York (1999)

\bibitem{Trushin'09} Trushin, D.V.: Difference Nullstellensatz. arXiv:math.AC/0908.3865

\bibitem{Zobnin'05} Zobnin, A.: Admissible Orderings and Finiteness
    Criteria for Differential Standard Bases. Kauers, M. (ed.) Proceedings of ISSAC'05, pp.\ 365--372. Association for Computing Machinery (2010)






\end{thebibliography}
\end{document}